\documentclass[letterpaper, 10 pt, conference]{ieeeconf}
\usepackage{mymacros}

\IEEEoverridecommandlockouts
\usepackage[left=1.9cm,right=1.9cm,top=1.9cm,bottom=1.9cm]{geometry}

\title{\LARGE \bf Automated Worst-Case Performance Analysis of\\Decentralized Gradient Descent
}
\author{Sebastien Colla and Julien M. Hendrickx
\thanks{S. Colla  and  J.  M.  Hendrickx  are  with  the  ICTEAM institute, UCLouvain, Louvain-la-Neuve, Belgium. S. Colla is supported by the French Community of Belgium through a FRIA fellowship (F.R.S.-FNRS). J. M. Hendrickx is supported by the “RevealFlight” Concerted Research Action (ARC) of the Federation Wallonie-Bruxelles and by the Incentive Grant for Scientific Research (MIS) “Learning from Pairwise Comparisons” of the F.R.S.-FNRS. Email addresses: {\tt\small sebastien.colla@uclouvain.be}, {\tt\small julien.hendrickx@uclouvain.be}}
}

\begin{document}
\maketitle
\thispagestyle{empty}
\pagestyle{empty}

\begin{abstract}
  We develop a methodology to automatically compute worst-case performance bounds for a class of decentralized algorithms that optimize the average of local functions distributed across a network.
  We extend the recently proposed PEP approach to decentralized optimization. This approach allows computing the exact worst-case performance and worst-case instance of centralized algorithms by solving an SDP.
  We obtain an exact formulation when the network matrix is given, and a relaxation when considering entire classes of network matrices characterized by their spectral range.
  We apply our methodology to the decentralized (sub)gradient method, obtain a nearly tight worst-case performance bound that significantly improves over the literature, and gain insights into the worst communication networks for a given spectral range.
\end{abstract}
\section{Introduction} \vspace{-1mm}
The goal of this paper is to develop a methodology that automatically provides nearly tight performance bounds for primal first-order decentralized methods on convex functions.

Decentralized optimization has received an increasing attention due to its useful applications in large-scale machine learning and sensor networks, see for example references in this survey \cite{DGD}. In decentralized methods for separable objective functions, we consider a set of agents $\{1,\dots,N\}$, working together to optimize this global objective: \vspace{-2mm}
\begin{equation} \label{opt:dec_prob}
  \underset{\text{\normalsize $x \in \Rvec{d}$}}{\mathrm{minimize}} \quad f(x) = \frac{1}{N}\sum_{i=1}^N f_i(x), \vspace{-1mm}
\end{equation}
where $f_i$ is the private function locally held by agent $i$.
To achieve this goal, each agent $i$ of the system holds its own version $x_i$ of the decision variable $x \in \Rvec{d}$. Agents perform local computations and exchange local information with their neighbors to seek to reach an agreement on the minimizer $x^*$ of the global function $f$. Exchanges of information often take the form of a multiplication by a given matrix $W \in \Rmat{N}{N}$, typically assumed symmetric and doubly stochastic.

A classical example of decentralized optimization method is the decentralized (sub)gradient descent (DGD) \cite{DsubGD} where agents successively perform the consensus step \eqref{eq:DGD_cons} and the local gradient step \eqref{eq:DGD_comp}. We have, for all $i\in \{1,\dots,N\}$, \vspace{-2mm}
\begin{align}
  y_i^k &= \sum_{j=1}^N w_{ij} x_j^k,           \hspace*{15mm} \label{eq:DGD_cons} \\[-1mm]
  x_i^{k+1} &= y_i^k - \alpha \nabla f_i(x_i^k), \hspace*{15mm} \label{eq:DGD_comp} \vspace{-2mm}
\end{align}
where $\alpha > 0$ is a constant step-size.
Although the tools we develop in this paper are general, this simple algorithm is used as a case study in Section \ref{sec:NumRes}. This focus has been chosen for the simplicity of the algorithm and not for its performance. Indeed there exists many other decentralized algorithms that perform better, including EXTRA \cite{EXTRA}, DIGing \cite{DIGing}, NIDS \cite{NIDS}.
The analysis of these other algorithms will be the focus of future work.

The quality of an optimization method is often evaluated via a worst-case guarantee.
Obtaining theoretical worst-case performance bounds for decentralized algorithms can often be a challenging task, requiring combining the impact of the optimization component and the interconnection network. This can result in performance bounds that are not very tight.
For example, we will show in Section \ref{sec:NumRes} that the available performance bounds of DGD are significantly worse than the actual worst-cases.

In this work, we follow an alternative computational approach that finds a worst-case performance guarantee of an algorithm by solving an optimization problem. This is known as the performance estimation problem - PEP - and has been studied for centralized fixed-step first-order methods, see e.g. \cite{PEP_Smooth,PEP_composite}. PEP has never been applied to decentralized optimization methods. Particularly, the current PEP framework does not allow to represent matrix multiplications in the methods it analyzes, except if the matrix is explicitly given. This paper proposes possible solutions for this missing piece to the analysis of decentralized methods via PEP. Our contributions are the following:

We provide two formulations of the multiplication by a symmetric \emph{generalized doubly stochastic} matrix $W$, which is defined as a doubly stochastic matrix but without the restriction of being non-negative. Those formulations allow to write and solve PEP for a large class of decentralized methods. The first formulation is specific to a communication matrix $W$ given a priori, is exact, and is directly derived from the current PEP framework. It can be applied to any matrix $W$ and leads to tight performance bounds that are specific to this given matrix.
The second formulation is a relaxation that considers entire classes of possible matrices, based on their spectrum. It is our main methodological contribution.
We demonstrate the usefulness of these new formulations of PEP by analyzing the worst-case performance of DGD. For DGD, the second relaxed formulation leads to tight spectral performance bounds significantly improving on the existing theoretical ones. Our numerical experiments show that our bounds are independent of the number of agents $N$ and can be used to easily choose the optimal step-size of the method.

\section{General PEP approach}
In principle, a tight performance bound on an algorithm could be obtained by running it on every single instance - function and initial condition - allowed by the setting considered and selecting the worst performance obtained. This would also directly provide an example of “worst” instance if it exists.
The performance estimation problem (PEP) formulates this abstract idea as a real optimization problem that maximizes the error measure of the algorithm result, over all possible functions and initial conditions allowed. This optimization problem is inherently infinite-dimensional, as it contains a continuous function among its variables. Nevertheless, Taylor et al. have shown \cite{PEP_Smooth,PEP_composite} that PEP can be solved exactly for a wide class of centralized first-order algorithms on convex functions, using an SDP formulation.

We illustrate this approach with a simple example on $K$ steps of the centralized unconstrained subgradient descent. Let Perf denotes any performance measure for which we would like to find the worst-case, e.g. $f(x^K) - f(x^*)$. Perf can depend on the function $f$, its (sub)gradients, the minimizer $x^*$ and any of the iterates $x^k$. PEP for this algorithm can be written as follows: \vspace{-1mm}
\begin{align}
  \underset{f, x^0,\dots, x^K, x^*}{\sup} \quad & \mathrm{Perf}\qty(f, x^0,\dots,x^K, x^*) \hspace*{-4mm}\tag{Subgradient-PEP} \label{prob:GD_PEP}\\[-1mm]
 \text{s.t.} \hspace{8mm}       & \hspace{-5mm}    f \in \mc{F} \label{eq:fct_class}\\[-0.6mm]
                                & \hspace{-5mm} x^* = \underset{x}{\mathrm{argmin}}~ f(x), \\[-0.6mm]
                                & \hspace{-5mm} x^{k+1} = x^{k} - \alpha \nabla f(x^k), \quad \text{for } k=0,\dots,K-1, \\[-0.6mm]
                                & \hspace{-5mm} x^0 \text{ satisfies some initial conditions.} \\[-6.5mm]
\end{align}
where $\mc{F}$ denotes a class of functions and $\nabla f(x^k)$ denotes any subgradient of $f$ at $x^k$. \\
To overcome the infinite dimension of variable $f$, we notice that the problem \eqref{prob:GD_PEP} only uses the values and subgradients of the function at specific points: the iterates $x^0,\dots,x^K$ and the minimizer $x^*$.
This motivates the discretization of the problem: the decision variables are restricted to the iterations, subgradients, and function values associated with these specific points $\{\qty(x^k,g^k,f^k)\}_{k\in I}$ and we add the constraint that there is a function of the class $\mc{F}$ which interpolates those data points $\{\qty(x^k,g^k,f^k)\}_{k \in I}$.
This can be done using necessary and sufficient interpolation constraints for the functional class under consideration. Such constraints are provided for many different classes of functions in \cite[Section 3]{PEP_composite}.
For example, for the class of convex functions with bounded subgradients $\mc{F}_{B}$, we can use interpolation constraints from the following particularization of \cite[Theorem 3.5]{PEP_composite}, initially formulated for smooth convex functions with bounded gradient. \smallskip
\begin{theorem}[\hspace{-0.5pt}{{\cite[Theorem 3.5]{PEP_composite}}}]\label{thm:interp}
  Let $I$ be an index set.
  There exists a function $f \in \mc{F}_{B}$ such that $f^k = f(x^k)$ and $g^k = \nabla f(x^k)$ ($k\in I$) if and only if for all pair $k \ne l \in I$ \vspace{-1mm}
  \begin{align} \label{eq:interp_cond}
    \quad f^k &\ge f^l + \scal{g^l}{x^k-x^l}, \quad \text{and} \quad
    ||g^k||^2 &\le B^2. \\[-4mm]
  \end{align}
\end{theorem}

As it is the case for $\mc{F}_B$, the interpolation constraints are generally quadratics, potentially non-convex, involving scalar products and functional values.
They can be reformulated using a Gram matrix $G$ and a vector containing function values $f = [f_i]_{i \in I}$.
The Gram matrix is a symmetric positive semidefinite matrix containing scalar products between iterates $x^k \in \Rvec{d}$ and subgradients $g^k \in \Rvec{d}$ for $k \in I$. \vspace{-0.5mm}
\begin{align}
  G = P^TP, \text{ with } P = \qty[g^0\dots g^K g^* x^0\dots x^K x^*]. \\[-6mm]
\end{align}
Quadratics interpolations constraints from Theorem \ref{thm:interp} are linear in $G$ and $f$. Linear equality constraints, such as the constraints for the iterates of the algorithm, can also be expressed linearly with $G$, by isolating all elements on the same side and squaring the equation. This leads to an equivalent positive semidefinite program (SDP) for PEP with the Gram matrix $G$ and the vector of functional values $f$ as variables. Besides the reformulation of interpolation, optimality, iterates, and initial constraints, we should also impose that $G \succeq 0$ and $\text{rank } G \le d$. By relaxing this rank constraint, the formulation becomes independent of the dimension $d$ of the iterates and provides the worst case in any dimension.
This SDP formulation is convenient because it can be solved numerically to global optimality.
See \cite{PEP_composite} for details about the SDP formulation of PEP, including ways of reducing the size of matrix $G$. However, the dimension of $G$ always depends on the number of iterations $K$.

From a solution $G$, $f$ of the SDP formulation, we can construct, using Cholesky decomposition for example, a solution for the discretized variables $\{\qty(x^k,g^k,f^k)\}_{k\in I}$. Since these points satisfy sufficient interpolation constraints, we can also construct a function from $\mc{F}$ interpolating these points.

The following proposition states a sufficient condition under which a PEP can be formulated as an SDP, which can then be solved exactly. \smallskip
\begin{definition}[Gram-representable]
Consider a Gram matrix $G$ and a vector $f$, as defined previously. We say that
a constraint or an objective is linearly (resp. LMI) Gram representable if it can be expressed using a finite set of linear (resp. LMI) constraints involving (part of) $G$ and $f$.
\end{definition} \smallskip
\begin{proposition}[\hspace{-0.5pt}{{\cite[Proposition 2.6]{PEP_composite}}}] \label{prop:GramPEP}
  If the interpolation constraints of the class of functions $\mc{F}$, the satisfaction of the method $\mc{M}$, the performance measure Perf and the set of constraints $\mc{I}$, which includes the initial conditions, are linearly (or LMI) Gram representable, then, computing the worst-case for criterion Perf of method $\mc{M}$ after $K$ iterations on objective functions in class $\mc{F}$ with constraints $\mc{I}$ can be formulated as an SDP, with $G$ and $f$ as variables. \\
  This remains valid when the objective function is the sum of $N$ sub-functions.
\end{proposition}

\emph{Remark:}
Proposition 2.6 in \cite{PEP_composite} was only formulated for linearly Gram-representable constraints, but its extension to LMI Gram-representable constraints is direct. Such constraints appear in the analysis of classes of network matrices.

PEP techniques allowed answering several important questions in optimization, see e.g. the list in \cite{Taylor_thesis}, and to make important progress in the tuning of certain algorithms including the well-known centralized gradient-descent. Following a numerical exploration in \cite{PEP_Drori}, it was further exploited by Kim and Fessler to design the Optimized Gradient Method (OGM), the fastest possible first-order optimization method for smooth convex functions \cite{OGM}.
It can also be used to deduce proofs about the performance of the algorithms \cite{PEP_compo}. It has been made widely accessible via a Matlab toolbox \cite{PESTO}.
However, PEPs have never been used to study decentralized methods. The main challenge is that
there is no representation of the interconnections between the agents that can be embedded in the formulation.

We also note an alternative approach with similar motivations %
for automated performance evaluation that is proposed in \cite{IQC}, and is inspired by dynamical systems concepts. Integral quadratic constraints (IQC), usually used to obtain stability guarantees on complex dynamical systems, are adapted in order to obtain sufficient conditions for the convergence of optimization algorithms. It provides iteration-independent linear rates of convergence, based on relatively small size problems, but it does not apply when the convergence is not geometric. Unlike PEP, it offers no a priori guarantee of tightness, though it turns out to be tight in certain situations.
An application of the IQC methodology to decentralized optimization is presented in \cite{IQC_dec} and has already been used for designing a new algorithm that achieves a faster worst-case convergence rate. But this methodology cannot be directly applied to DGD, nor to smooth convex functions or any other situation that does not have a geometric convergence. Also, this IQC approach focuses on one iteration of the algorithm to derive the worst-case convergence rate, and hence, it cannot exploit situations where the communication matrix is identical at each iteration to improve it.

\section{Representations of consensus steps for PEP}
In this section, we present a way of representing the interactions between agents and thereby providing the missing block to PEP formulation for first-order decentralized optimization. We focus on the situation where the interactions take place via a weighted averaging and can thus be described as a consensus step of the following form: \vspace{-2mm}
\begin{equation} \label{eq:consensus}
    y_i = \sum_{j=1}^N w_{ij} x_j, \qquad \text{for all $i\in\{1,\dots,N\}$,} \vspace{-2mm}
\end{equation}
where $x_j$ can represent any vector in $\Rvec{d}$ held by agent $j$, e.g. its local iterates in the case of DGD but it could be more evolved. Vector $y_i \in \Rvec{d}$ is an auxiliary variable that represents the result of the interaction and $W \in \Rmat{N}{N}$ is the weighted communication matrix.
While we focus on DGD in section IV, this form of communication is used in many other decentralized methods such as EXTRA \cite{EXTRA}, DIGing \cite{DIGing}, NIDS \cite{NIDS} and the results presented in this section can be exploited for all these methods too. \\
We suppose the communication matrix $W$ symmetric, i.e. $w_{ij} = w_{ji}$ and \emph{generalized doubly stochastic} \cite{gds, gds_first}, i.e. $\sum_{i=1}^N w_{ij} = 1$, $\sum_{j=1}^N w_{ij} = 1$, for all $i,j$.
This last assumption is a relaxation of the more usual double stochasticity since it does not require elements of $W$ to be non-negative. However, many results from the literature on decentralized optimization are based on spectral information and do not use the non-negativity of $W$ either, see e.g. \cite{DGD, EXTRA, NIDS}. We analyze the impact of this relaxation in the case of DGD in Section \ref{sec:NumRes}.

\subsection{Communication matrix given a priori}
When the communication matrix is given \emph{a priori}, the consensus step \eqref{eq:consensus} is a simple linear equality constraint that is linearly Gram-representable. In that case, the constraint can be used in the SDP formulation of a PEP, see Proposition \ref{prop:GramPEP}. This allows writing PEPs that provide exact worst-case performances for the given decentralized method and the specific communication matrix given \emph{a priori}. It can be applied to any matrix $W$, and not only for generalized doubly stochastic ones.
This can be useful for trying different communication matrices and observing their impact on the worst-case performance of the algorithm. It will serve as an exact comparison baseline in the numerical experiments of Section \ref{sec:NumRes}.
The next section presents a way of representing communications in PEP that allows obtaining more general performance guarantees, valid over entire classes of communication matrices and not only for a specific one.

\subsection{Communication matrix as a variable}
We now consider that the matrix $W$ is not given \emph{a priori}, but is \emph{one of the decision variables} of the performance estimation problem with bounds on its possible eigenvalues. Hence the search space contains the matrix $W$ in addition to the usual variables, and is restricted by the following constraints, for each agent $i\in\{1,\dots,N\}$ and each consensus steps $k\in\{1,\dots,K\}$, \vspace{-0.5mm}
\begin{align}
  W &\in \Wcl{\lm}{\lp}, \label{eq:cons_1} \\[-0.5mm]
  y_i^k &= \sum_{j=1}^N w_{ij} x_j^k, \label{eq:cons_2} \\[-6mm]
\end{align}
where $\Wcl{\lm}{\lp}$ is the set of real, symmetric and generalized doubly stochastic matrices that have their eigenvalues between $\lm$ and $\lp$, except for $\lam_1 = 1$:
$$ \lm \le \lam_N\le\cdots\le\lam_2\le\lp \quad \text{where $\lm, \lp \in \qty[-1,1]$ }. \vspace{-1mm}$$
The set of the different consensus steps represented by indices $k=1,\dots,K$ can for example correspond to the set of the different iterations of the algorithm, but it could also be a subset of the iterations if the communication matrix changes, or other subsets of the consensus steps if different consensus \eqref{eq:cons_2} are applied to the same iteration.

We do not have a direct way for representing constraints \eqref{eq:cons_1} and \eqref{eq:cons_2} in an LMI Gram-representable manner, but we construct a relaxation that we will see in section \ref{sec:NumRes} is often close to tight. From constraints \eqref{eq:cons_1} and \eqref{eq:cons_2}, we derive new necessary conditions involving only variables $y_i^k$ and $x_i^k$ allowing to eliminate $W$ from the problem.

We first restrict ourselves to the simpler case when the local variables are one-dimensional: $x_i^k, y_i^k \in \Rvec{d}$ with $d=1$.

Let $X$ and $Y$ be the following matrices from $\Rmat{N}{K}$:
$$ X_{ik} = x_i^k \quad \text{and}\quad Y_{ik} = y_i^k  \qquad \text{for ~~$\substack{i=1,\dots,N \\ k=1,\dots,K}$ }$$
Each column corresponds thus to a different consensus step $k$, and each row to a different agent $i$. Using this notation, the consensus steps constraints \eqref{eq:cons_2} can simply be written as $Y = W X$.
We decompose matrices $X$ and $Y$ in average and centered parts: \vspace{-1mm}
$$ X = \mathbf{1} \Xb^T + \Xc, \qquad Y =  \mathbf{1}\Yb^T + \Yc,$$
where $\Xb$ and $\Yb$ are agents average vectors in $\Rvec{K}$, defined as $\Xb_k = \frac{1}{N} \sum_{i=1}^N x_i^k$, $\Yb_k = \frac{1}{N} \sum_{i=1}^N y_i^k$ for $k=1,\dots,K$ and $\mathbf{1} = \qty[1\dots 1]^T$.
Using these notations, we state the new necessary conditions in the following theorem.

\begin{theorem}[Consensus Constraints] \label{thm:conscons}
If $Y = WX$ for a matrix $W\in \Wcl{\lm}{\lp}$, with $X, Y \in \Rmat{N}{K}$, then
\begin{enumerate}[(i)]
  \item The matrices $X^TY$ and $\Xc^T\Yc$ are symmetric,
  \item The following constraints are satisfied \vspace{-0.5mm}
  \begin{align}
    \Xb &= \Yb,  \label{eq:eq_mean} \\
    \lm \Xc^T \Xc ~ \preceq  ~\Xc^T \Yc~ &\preceq ~\lp \Xc^T \Xc, \label{eq:scal_cons} \\
    (\Yc - \lm \Xc)^T(\Yc - \lp \Xc) ~&\preceq ~0, \hspace{20mm} \label{eq:var_red} \\[-7mm]
  \end{align}
  where the notations $\succeq$ and $\preceq$ denote respectively positive and negative semi-definiteness.
  \item Constraints  \eqref{eq:eq_mean}, \eqref{eq:scal_cons},  \eqref{eq:var_red} are LMI Gram-representable.
\end{enumerate}
\end{theorem}
\begin{proof}
    First, we average elements from both sides of the assumption $Y = WX$ to obtain constraint \eqref{eq:eq_mean}: \vspace{-1mm}
    $$ \Yb = \frac{\mathbf{1}^T Y}{N} = \frac{\mathbf{1}^T WX}{N} = \frac{\mathbf{1}^T X}{N} = \Xb. \vspace{-1mm} $$
    where $\mathbf{1}^TW = \mathbf{1}^T$ follows from $W$ being generalized doubly stochastic, i.e. its rows and columns sum to one. \\
    The symmetry of the matrix $X^TY$ is directly obtained using the assumption $Y = WX$, with $W$ being symmetric. We can use the same argument for the symmetry of $\Xc^T\Yc$, because having $Y = WX$ and $\Xb = \Yb$ implies that $\Yc = W\Xc$. \\
    Since the communication matrix $W$ is real and symmetric, we can take an orthonormal basis $\mathbf{v_1},\dots,\mathbf{v_N}$ of eigenvectors, corresponding to real eigenvalues
    $ \lam_N\le\cdots\le\lam_2 \le \lam_1$.
    The largest eigenvalue is $\lambda_1 = 1$ and corresponds to the eigenvector $\mathbf{v_1} = \mathbf{1}$. Indeed $W$ is doubly-stochastic and have all its other eigenvalues below 1 by assumption:
    $$\lm \le \lam_i \le \lp \text{ for $i=2,\dots,N$, with $\lm, \lp \in [-1,1]$.} $$
    Let us now consider a combination $\Xc z$ of the columns of the matrix $\Xc$, for an arbitrary $z \in \Rvec{K}$. It can be decomposed in the eigenvectors basis of $W$, and used to express the combination $\Yc z$: \\[-8.5mm]

    \small
    \begin{align} \label{eq:ev_basis}
      \hspace{-2mm} \Xc z = 0 \mathbf{v_1} + \sum_{i = 2}^N \gamma_i \mathbf{v}_i, \text{ and }
      \Yc z = W \Xc z = \sum_{i = 2}^N \gamma_i \lambda_i \mathbf{v}_i, \hspace{4mm}  \\[-6.5mm]
    \end{align}
    \normalsize
    where $\gamma_i$ are real coefficients. The coefficient for $\mathbf{v_1} = \mathbf{1}$ is zero because $\Xc z$ is orthogonal to this eigenvector since it is centered: $\mathbf{1}^T(\Xc z) = 0$.
    Using this decomposition to compute the scalar product $z^T \Xc \Yc z$ for any $z \in \Rvec{K}$ leads to the following scalar inequalities \vspace{-2mm}
    \begin{align}
      z^T \Xc^T \Yc z = \sum_{i = 2}^N \gamma_i^2 \lambda_i &\ge  \lm z^T \Xc^T \Xc z, \\[-4mm]
                                                            &\le  \lp z^T  \Xc^T \Xc z.
    \end{align}
    Having these inequalities satisfied for all $z \in \Rvec{K}$, is equivalent to \eqref{eq:scal_cons}.
    In the same way, \eqref{eq:var_red} is obtained by verifying that the following inequality hold for all $z \in \Rvec{K}$:
    $$ (\Yc z - \lm \Xc z)^T(\Yc z - \lp \Xc z) \le 0. $$
    This can be done by substituting $\Xc z$ and $\Yc z$ using equation \eqref{eq:ev_basis}, and by using the bounds on $\lam_i$ ($i=2,\dots,N$). \\
    Finally, constraints \eqref{eq:eq_mean}, \eqref{eq:scal_cons} and  \eqref{eq:var_red} can all be formulated as LMIs involving submatrices of the Gram matrix $G$ of scalar products. Therefore, they are LMI Gram-representable.
\end{proof}
\smallskip
Using Theorem \ref{thm:conscons}, we can relax constraints \eqref{eq:cons_1} and \eqref{eq:cons_2} and replace them by \eqref{eq:eq_mean}, \eqref{eq:scal_cons} and  \eqref{eq:var_red}, which are LMI Gram-representable.
Then, Proposition \ref{prop:GramPEP} allows to write a relaxed SDP formulation of a PEP providing worst-case results valid for the entire spectral class of matrices $\Wcl{\lm}{\lp}$.

Constraint \eqref{eq:eq_mean} is related to the stochasticity of the communication matrix and imposes that variable $x$ has the same agents average as $y$, for each consensus step.
Linear matrix inequality constraints \eqref{eq:scal_cons} and  \eqref{eq:var_red} imply in particular equivalent scalar constraints for the diagonal elements.
They corresponds to independent constraints for each consensus step, i.e. for each column $\xc$ and $\yc$ of matrices $\Xc$, $\Yc$: \vspace{-0.5mm}
\begin{align}
  \lm \xc^T \xc \le \xc^T\yc &\le \lp \xc^T \xc, \label{eq:scal_cons_ind}\\
  (\yc - \lm \xc)^T(\yc - \lp \xc) & \le 0. \label{eq:var_red_ind} \\[-6.5mm]
\end{align}
These constraints imply in particular that \vspace{-1mm} $$\yc^T \yc \le \lmax^2 \xc^T\xc,\quad \text{where $\lmax = \max(|\lm|,|\lp|)$,} \vspace{-1mm}$$
meaning that the disagreement between the agents, measured by $\yc^T \yc$ for $y$ and $\xc^T \xc$ for $x$, is reduced by a factor $\beta^2$ after a consensus.
But constraints \eqref{eq:scal_cons} and  \eqref{eq:var_red} also allow linking different consensus steps to each other, via the impact of off-diagonal terms, in order to exploit the fact that these steps use the same communication matrix. Moreover, if different communication matrices are used for different sets of consensus steps, the constraints from Theorem \ref{thm:conscons} can be applied independently for each set of consensus steps.

Theorem \ref{thm:conscons} is derived for $x_i^k, y_i^k \in \R$ but the worst-case solution of PEP has no guarantee to be one-dimensional. When the local variables are multi-dimensional: $x_i^k, y_i^k\in\Rvec{d}$ with $d\ge1$, a natural approach would be to impose the constraints independently for each dimension. This might introduce some conservatism in our relaxed formulation, as it would allow each dimension to use a different matrix. But independently of this issue, this approach cannot be directly implemented in PEP for constraints \eqref{eq:scal_cons} and \eqref{eq:var_red} because we cannot access the different dimensions of a variable in the SDP formulation which only allows using scalar products.
One solution is to sum these constraints over all dimensions. This may lead to the constraints not being exactly met for a specific dimension and may thus also introduce conservatism in our relaxed formulation.
The following corollary presents the resulting constraints after the sum on the dimensions.
The matrices $X^j, Y^j \in \Rmat{N}{K}$ contains the element from dimension $j$ of each $x_i^k$, $y_i^k$ variables and the matrices $X, Y \in \Rmat{Nd}{K}$ stack each $X^j$ and $Y^j$ vertically. \smallskip
\begin{corollary}[Theorem \ref{thm:conscons} for $d\ge1$] \label{cor:consD}
  If $Y^j = WX^j$, for every $j=1,\dots,d$ and for a same matrix $W\in \Wcl{\lm}{\lp}$,
  i.e. $Y = (I_d \otimes W) X$, where $\otimes$ is the Kronecker product, then
  \begin{enumerate}[(i)]
    \item The matrices $X^TY$ and $\Xc^T\Yc$ are symmetric,
    \item Constraints \eqref{eq:eq_mean}, \eqref{eq:scal_cons} and \eqref{eq:var_red} are satisfied.
    \item Constraints \eqref{eq:eq_mean}, \eqref{eq:scal_cons}, \eqref{eq:var_red} are LMI Gram-representable.
  \end{enumerate}
\end{corollary}
The proof of this corollary will be provided in the journal version of this article.

\section{Decentralized (sub)gradient descent (DGD): \\ A case study}
\label{sec:NumRes}
Using results from previous sections, we can build two PEP formulations for analyzing the worst-case performance of the well-known decentralized gradient descent. We consider $K$ iterations of DGD described by \eqref{eq:DGD_cons} and \eqref{eq:DGD_comp}, with constant step-size, in order to solve problem \eqref{opt:dec_prob}, i.e. minimizing $f = \frac{1}{N}\sum_{i=1}^N f_i(x)$, with $x^*$ as minimizer of $f$.
There are different studies about DGD; e.g. \cite{DGD1} shows that its iterates converge to a neighborhood of the optimal solution $x^*$ when the step-size is constant. In the following analysis, we consider the results of a recent survey \cite{DGD} providing a theoretical bound for the functional error at the average of all the iterates. This error tends to zero since it considers the average of all the iterates and not only the last one. \smallskip

\begin{theorem}[Performance of DGD {{\cite[Theorem 8]{DGD}}}] \label{thm:bound}
Let $f_i,\dots,f_N$ be convex local functions with subgradients bounded by $B$. Let $x^0$ be an identical starting point of each agent such that $||x^0 - x^*||^2 \le R^2$. And let $W$ be a symmetric and doubly stochastic matrix with eigenvalues $\lam_2,\dots,\lam_N \in \qty[-\lam,\lam]$, for some $\lam \in [0,1)$. \\
If we run DGD for $K$ steps with a constant step-size $\alpha = \frac{1}{\sqrt{K}}$, then there holds\footnote{Note the factor 2 in the second term of the bound \eqref{eq:th_bound} was missing in \cite{DGD} but its presence was confirmed by the authors of \cite{DGD}.}
\vspace{-3pt}
\begin{equation} \label{eq:th_bound}
  f(\xmoy ) - f(x^*) \le \frac{R^2 + B^2}{2 \sqrt{K}} + \frac{2 B^2}{\sqrt{K}(1-\lam)}, \vspace{-3pt}
\end{equation}
where $\xmoy = \frac{1}{N(K+1)} \sum_{i=1}^N \sum_{k=0}^K x_i^k $ is the average over all the iterations and all the agents.
\end{theorem}

For comparison purposes, our PEP formulations of DGD use the same assumptions as Theorem \ref{thm:bound}.
Our first formulation characterizes the same performance measure as Theorem \ref{thm:bound}: $f(\xmoy)-f(x^*)$, with appropriate constraints for the initialization: $||x^0 - x^*||^2 \le R^2$, the set of functions (Theorem \ref{thm:interp}) and the algorithm iterates that are generated by DGD with a given matrix $W$ (equations \eqref{eq:DGD_cons} and \eqref{eq:DGD_comp}).
All these constraints are linearly Gram-representable and can then be used in the SDP formulation of PEP, according to Proposition \ref{prop:GramPEP}.
This formulation is referred to as the \emph{exact formulation} because it finds the exact worst-case performance of the algorithm in the specific case where it is used with the given matrix $W$.
The second formulation relaxes the consensus constraints imposed by equation \eqref{eq:DGD_cons} and replaces them with the constraints from Theorem \ref{thm:conscons}, with $-\lm=\lp = \lam$. Those are LMI Gram-representable (see Corollary \ref{cor:consD}) and can then be used in the SDP formulation of PEP, according to Proposition \ref{prop:GramPEP}. This formulation is referred to as the \emph{spectral formulation} and provides \emph{spectral worst-cases},
i.e. upper bounds on the exact worst-case performances of the algorithm, valid for any matrix $W \in \Wcl{-\lam}{\lam}$, i.e. for any symmetric generalized doubly stochastic matrix with a given range of eigenvalues. In particular, these spectral upper bounds also hold for non-negative matrices from $\Wcl{-\lam}{\lam}$ and can therefore be compared with the bound from Theorem \ref{thm:bound}.

In our experiments, we focus on the situation where $R = 1$ and $B = 1$, but the results obtained can be scaled up to general values of $R$ and $B$ using appropriate changes of variables, as explained in the appendix.
\smallskip

\paragraph*{Impact of the number of agents $N$}
In Fig. \ref{fig:wc_Nevol}, we observe that the results of the spectral formulation are \emph{independent} of the number of agents $N \ge 2$ there are in the problem. This is shown for $K$ = 5 iterations and for different spectral ranges.
This observation has been confirmed for other values of $K$ (10, 15, and 20).
Moreover, the theoretical performance bound from Theorem \ref{thm:bound} is also independent of $N$.
Therefore, in the sequel, we analyze the spectral formulation for $N=3$, which is the smallest value of $N$ that still allows non-trivial communication networks to be considered.

\begin{figure}[t]
  \vspace{0mm}
  \centering
  \includegraphics[width=0.475\textwidth]{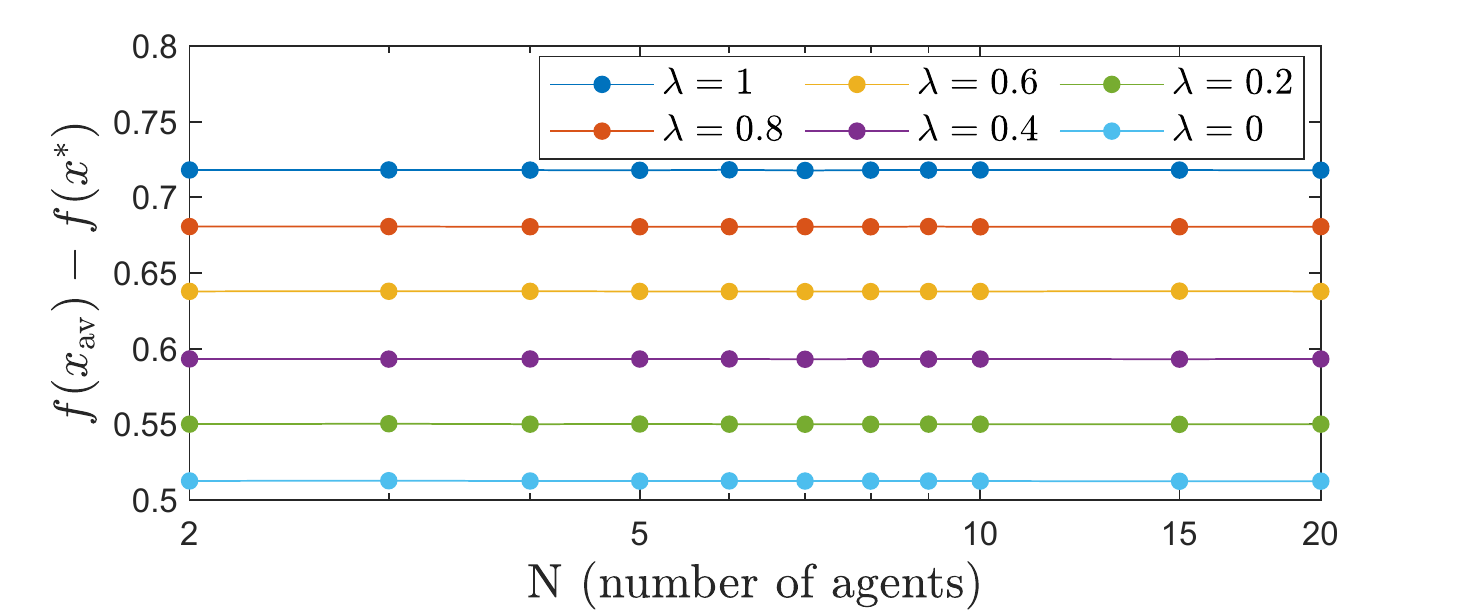}
    \vspace{-1mm}
    \caption{Independence of $N$ for the spectral worst-case performance of $K=5$ iterations of DGD in the setting~ of Theorem \ref{thm:bound}. \vspace{-5mm}}
  \label{fig:wc_Nevol}
\end{figure}
\smallskip
\paragraph*{Comparison with Theorem \ref{thm:bound}}
We compare the spectral bound with the theoretical bound from Theorem \ref{thm:bound} for different values of $\lam$.
Fig. \ref{fig:wc_lamevol_N3} shows the evolution of both bounds with $\lam$ for $K = 10$ iterations of DGD with $N = 3$ agents. We observe that the spectral worst-case performance bound (in blue) largely improves on the theoretical one (in red), especially when $\lam$ approaches 1, in which case the theoretical bound grows unbounded.
Moreover, having large values for $\lam$ is frequent for communication matrices of large networks of agents \cite{eigenBound}. Hence, the improvements of the bounds when $\lam$ is close to 1 are even more valuable.
For example, for a 5 by 5 grid of agents with Metropolis weights \cite{DGD}, the resulting communication matrix has eigenvalues $\lam_2,\dots,\lam_N \in \qty[-0.92,0.92]$.
In that case, after $K=10$ iterations, our spectral bound guarantees that the performance measure is below $0.85$, compared to 8.3 for the theoretical bound from Theorem \ref{thm:bound}.
This accuracy of $0.83$ would only be guaranteed using Theorem \ref{thm:bound} with $K = 950$. \smallskip

\paragraph*{Worst communication matrix and tightness analysis}
If the solution variables $X$ and $Y$ from the spectral PEP formulation can be linked by a unique matrix, then it can be found using the pseudo-inverse $\hat{W} = YX^+$. In practice, the matrix may not exist or may not be unique but $\hat{W}$ still provides an approximation of the worst-case matrix. When considering the spectral formulation with a symmetric spectral range $-\lm=\lp=\lam$ and $K$ large enough, we observe that it recovers, within numerical errors, matrices of the following form \vspace{-1.5mm}
\begin{equation} \label{eq:mat}
  [W^{(1)}]_{ij} = \begin{cases} \frac{1+\lambda}{N} & \text {if } i \ne j, \\ 1 - (N-1) \frac{1+\lambda}{N} & \text{if } i=j \end{cases}
\end{equation}
Matrix $W^{(1)}$ is symmetric and is generalized doubly stochastic, leading $1$ to be one of its eigenvalues. All its other eigenvalues are equal to $-\lam$. The bound obtained using the exact PEP formulation with this specific matrix $W^{(1)}$ for $K=10$ is plotted in green in Fig. \ref{fig:wc_lamevol_N3} and \emph{exactly} reaches the spectral bound in blue, within numerical errors. In that case, the spectral formulation, even though it is a relaxation, provides a \emph{tight performance bound} for DGD with symmetric generalized doubly stochastic matrices. This observation has been confirmed for other values of $K$ (5, 15, and 20). \vspace{1pt}

\paragraph*{Doubly stochastic versus generalized doubly stochastic}
Since every doubly stochastic matrix is also generalized doubly stochastic, the spectral bound also provides an upper bound on the performance of DGD with symmetric doubly stochastic matrices. This bound remains tight for $\lam \le \frac{1}{N-1}$ because the worst-case matrix $W^{(1)}$ \eqref{eq:mat} we have obtained is non-negative and is therefore doubly stochastic. For $\lam > \frac{1}{N-1}$,  this is no longer the case and the analysis is performed by empirically looking for symmetric stochastic communication matrices leading to the worst performance.
In Fig. \ref{fig:wc_lamevol_N3}, for $N=3$ and $\lam>0.5$, we have generated more than 6000 random symmetric doubly stochastic 3 by 3 matrices. We have analyzed their associated DGD performance using the exact PEP formulation and have only kept those leading to the worst performances.
The resulting pink curve deviates no more than 20\% below the spectral bound. In that case, the spectral bound is thus no longer tight for DGD with doubly stochastic matrices but remains very relevant.
This observation has been confirmed for other values of $K$ and $N$ ($N = 3,5,7$, and $K=10,15$). \vspace{1pt}

\paragraph*{Evolution with the total number of iterations $K$}
Fig. \ref{fig:wc_Kevol} shows the evolution of the spectral worst-case performance for DGD multiplied by $\sqrt{K}$, for different values of $\lam$. Except when $\lam = 1$, all lines tend to a constant value, meaning that the spectral bound behaves in $\bigO\qty(\frac{1}{\sqrt{K}})$, as the theoretical bound \eqref{eq:th_bound}, but with a much smaller hidden constant.
When $\lam = 1$, the line grows linearly and never reaches a constant value. In that case, the worst communication matrices lead to counterproductive interactions, preventing DGD from working in the worst case.

\paragraph*{Tuning the step-size $\alpha$} \vspace{1pt}
The PEP methodology allows us to easily tune the parameters of a method. For example, Fig. \ref{fig:wc_alphevol} shows the evolution of the spectral worst-case performance of DGD with the constant step-size it uses, in the setting of Theorem \ref{thm:bound} with $N=3$, $K=10$ and $\lam = 0.8$.
In that case, we observe that the value $\alpha = \frac{1}{\sqrt{K}}$ used in Theorem \ref{thm:bound} for deriving the theoretical performance bound is not the best possible choice for $\alpha$ and should be divided by two to improve the performance guarantees by 30\%.
The optimal value for $\alpha$, regarding our spectral bound, is the one that provides the best worst-case guarantee,
whatever the communication matrix from $\Wcl{-\lam}{\lam}$ is used.

The impact of the step-size on the other experiments and observations can be studied by setting $\alpha = \frac{h}{\sqrt{K}}$, for some $h > 0 $. We focused on $h=1$ for comparison with the theoretical bound from Theorem \ref{thm:bound}. Nevertheless, all our other observations have been confirmed for $h = 0.1, 0.5, 2, 10$.

\begin{figure}[h!]
  \vspace{-2mm}
  \centering
  \includegraphics[width=0.475\textwidth]{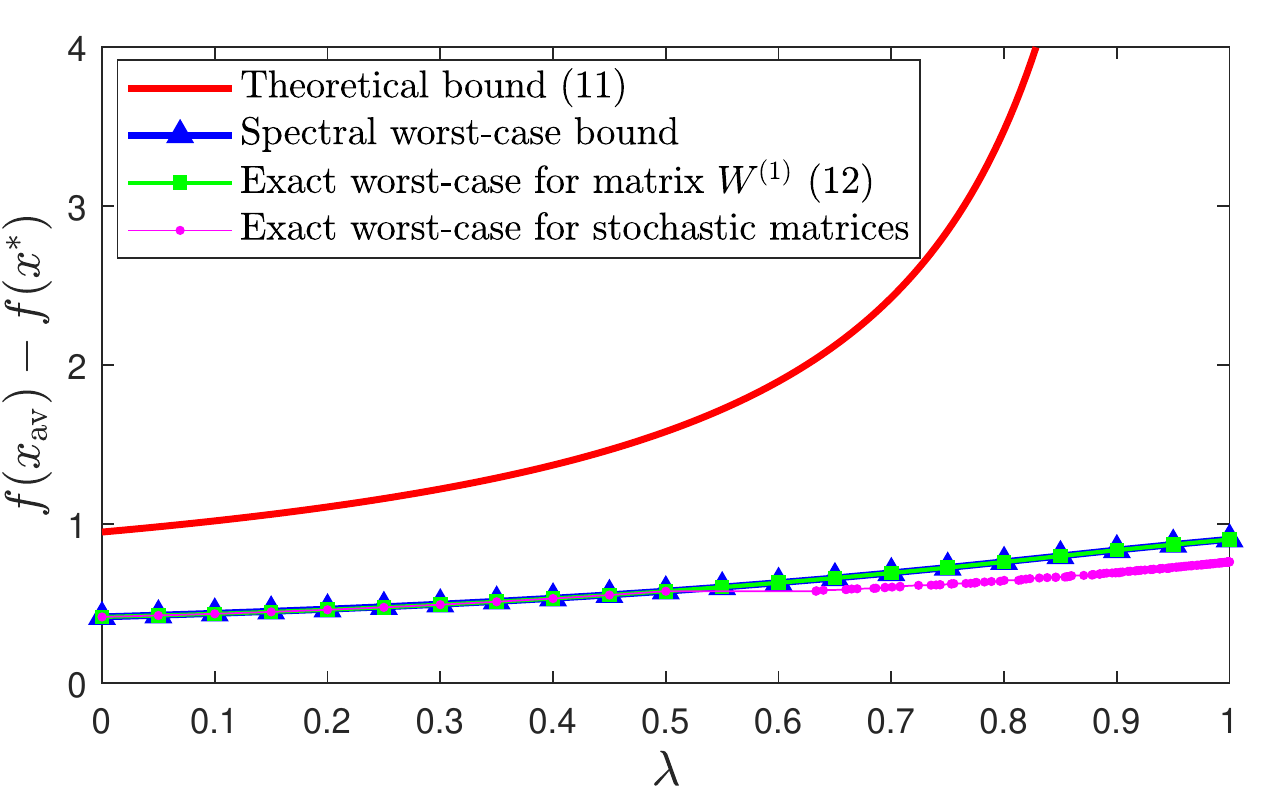}
  \vspace{-1mm}
  \caption{Evolution with $\lam$ of the worst-case performance of $K = 10$ iterations of DGD in the setting of Theorem \ref{thm:bound} with $N = 3$ agents. The plot shows (i) the theoretical bound from equation \eqref{eq:th_bound} (in red), largely above (ii) the spectral worst-case performance (in blue), (iii) the exact worst-case performance for the symmetric generalized doubly stochastic matrix $W^{(1)}$ from equation \eqref{eq:mat} (in green) and (iv) the exact worst-case performance for symmetric doubly stochastic matrices found based on an exhaustive exploration of such matrices used in the exact PEP formulation (in pink).
  This indicates the tightness of the spectral formulation of PEP for DGD with symmetric generalized doubly stochastic matrices, within numerical errors. \vspace{-3mm}
  }
  \label{fig:wc_lamevol_N3}
\end{figure}

\begin{figure}[h!]
  \vspace{-3mm}
  \centering
  \includegraphics[width=0.475\textwidth]{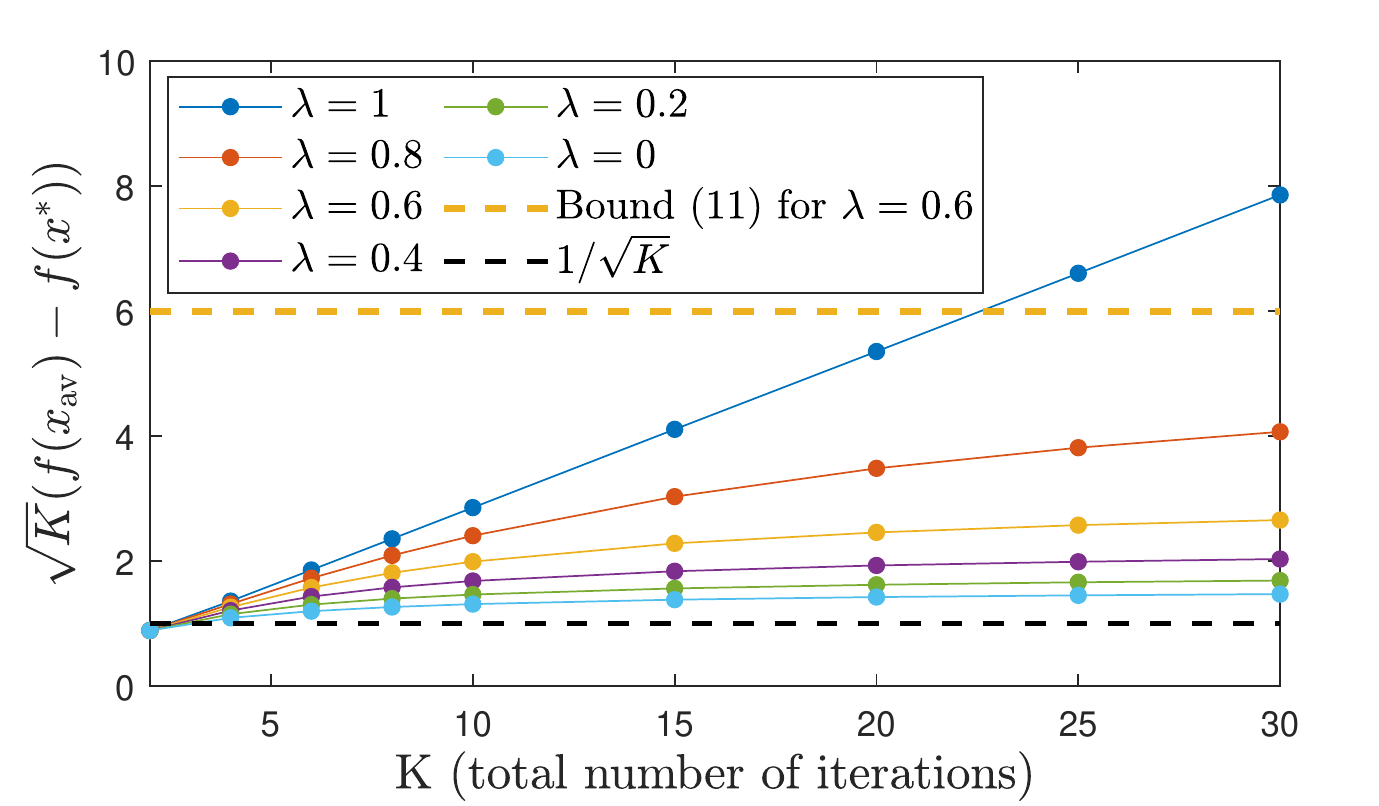}
 \vspace{-1mm}
  \caption{Evolution with $K$ of the \emph{normalized} spectral worst-case performance of $K$ iterations of DGD in the setting of Theorem \ref{thm:bound} with $N = 3$. The shown spectral worst-cases are normalized by $\frac{1}{\sqrt{K}}$ to show that they evolve at this rate.
  \vspace{-2mm}}
  \label{fig:wc_Kevol}
\end{figure}

\begin{figure}[h!]
  \vspace{-4mm}
  \centering
  \includegraphics[width=0.45\textwidth]{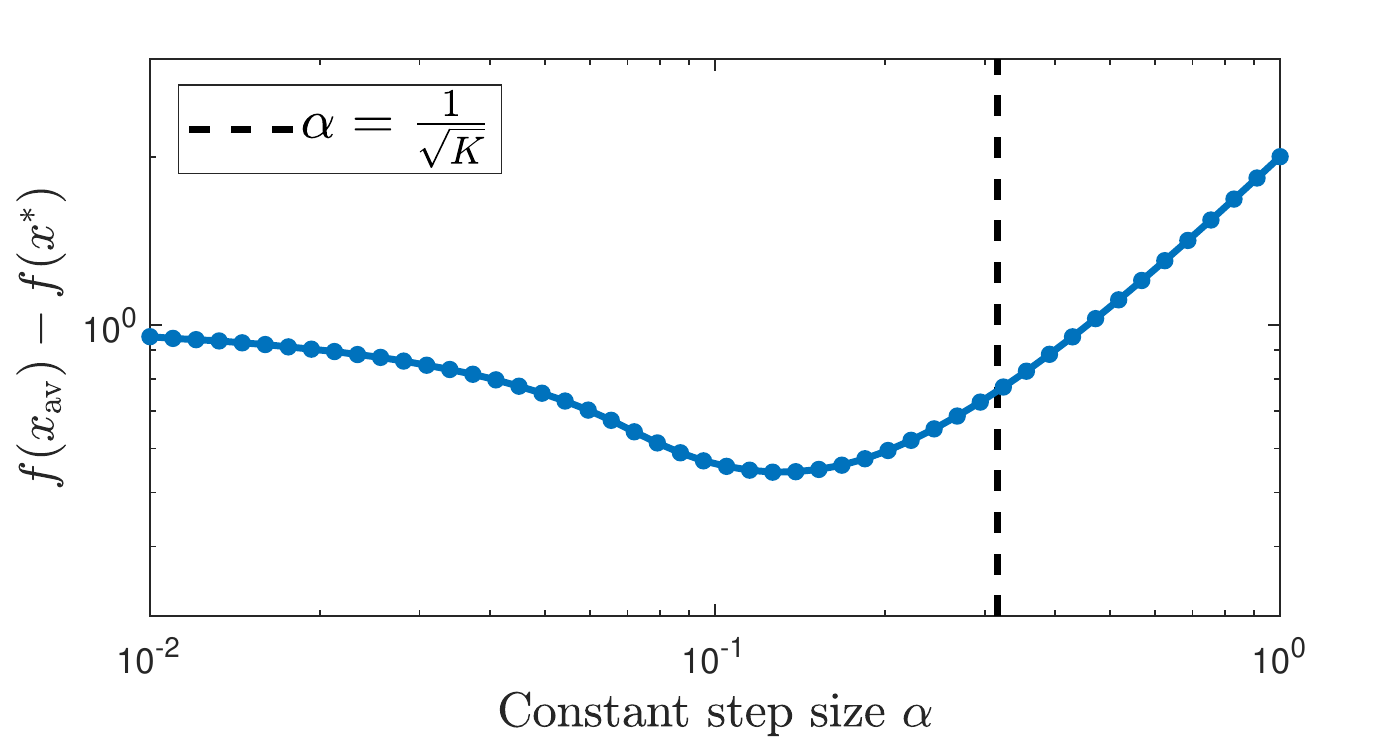}
  \vspace{-1mm}
  \caption{Evolution with $\alpha$ of the spectral worst-case performance of $K = 10$ iterations of DGD in the setting of Theorem \ref{thm:bound} with $N$ = 3 agents and $\lam = 0.8$ (except for $\alpha$). \vspace{-4mm}}
  \label{fig:wc_alphevol}
\end{figure}

\paragraph*{Code and Toolbox}
PEP problems are solved using the PESTO Matlab toolbox \cite{PESTO}, with Mosek solver, within 200 seconds. For example, for $N=3$, the time needed for a regular laptop to solve the spectral formulation is about 3, 12, 48, and 192 seconds respectively for $K=5, 10, 15, 20$.  The PESTO toolbox is available on \textsc{Github} (repository AdrienTaylor/Performance-Estimation-Toolbox). We have updated PESTO to allow easy and intuitive PEP formulation of gradient-type decentralized optimization methods, and we have added a code example for DGD.
\section{Conclusion}
We have developed two representations of consensus steps that can be embedded in the performance estimation problem (PEP) in order to automatically compute the worst-case performance of decentralized optimization methods in which such consensuses appear.
Our first formulation uses a given communication matrix to directly incorporate
the updates of the chosen method as constraints over the iterates. It provides the exact worst-case performance of the method for this specific matrix.
The second formulation provides upper bounds on the worst-case performance that are valid for an entire spectral class of matrices. It relaxes the specific consensus constraints and adds new necessary constraints for the given spectral class of matrices.
Although the second formulation is a relaxation, the performance guarantees it provides for DGD largely improve on the theoretical existing ones and are numerically tight for the class of symmetric generalized doubly stochastic communication matrices. Moreover, these resulting spectral bounds for DGD are independent of the number of agents in the problem and help for better tuning of the step-size.
Further works could exploit this new spectral formulation to analyze, develop our understanding, and improve many other decentralized methods.
Finally, this work could also help in the creation of new decentralized methods.

\section*{Acknowledgments}
The authors wish to thank Adrien Taylor for his helpful advice concerning the PESTO toolbox.

\bibliographystyle{IEEEtran}
\bibliography{refs.bib}

\appendix[Note on scaling]
\label{annexe:scaling}
We consider general positive values for parameters $R > 0$ and $B > 0$ and we parametrize the step-size by $\alpha = \frac{Rh}{B\sqrt{K}}$, for some $h > 0$. To pass from this general problem to the specific case where $R=1$ and $B=1$, we consider the following changes of variables:
$$\tilde{x} = \frac{x}{R}, \quad \tilde{f}(\tilde{x}) = \frac{1}{RB} f(x) \quad \text{ and } \quad \tilde{\alpha} = \frac{B \alpha}{R}.$$
These changes of variables do not modify the nature of the problem and allow expressing the worst-case guarantee obtained for $f(\xmoy ) - f(x^*)$ with general values of $R$, $B$, and $h$, denoted $w(R,B,h)$, in terms of the worst-case guarantee obtained for $\tilde{f}(\tilde{x}_{\mathrm{av}}) - \tilde{f}(\tilde{x}^*)$ with $R=B=1$:
\begin{equation} \label{eq:scal_wc}
  w(R,B,h) = RB~ \tilde{w}(1,1,h).
\end{equation}
The same kind of scaling can be applied to Theorem \ref{thm:bound}. The theorem is valid for general values of $R$ and $B$ but is specific to $\alpha = \frac{1}{\sqrt{K}}$, which is equivalent to picking $h = \frac{B}{R}$. After the scaling, we obtain the following bound, valid for $R = 1$, $B = 1$ and any value of $\alpha = \frac{h}{\sqrt{K}}$ with $h>0$:
\begin{equation} \label{eq:th_bound_scal}
  \tilde{f}(\tilde{x}_{\mathrm{av}} ) - \tilde{f}(\tilde{x}^*) \le \frac{h^{-1} + h}{2 \sqrt{K}} + \frac{2 h}{\sqrt{K}(1-\lam)}.
\end{equation}
This scaled theoretical bound with $h=1$ is equivalent to the bound from Theorem \ref{thm:bound} with $R=B=1$, which was the focus of the numerical analysis in Section \ref{sec:NumRes}. \\
This bound \eqref{eq:th_bound_scal} can be extended to any value of $R > 0$ and $B > 0$, using the relation from equation \eqref{eq:scal_wc}:
\begin{equation} \label{eq:th_bound_scal_2}
  f(\xmoy ) - f(x^*) \le RB \qty(\frac{h^{-1} + h}{2 \sqrt{K}} + \frac{2 h}{\sqrt{K}(1-\lam)}).
\end{equation}

\end{document}